\documentclass{daj}

\usepackage{siunitx}
\usepackage{amssymb}
\usepackage{amsmath}
\usepackage{amsthm}
\usepackage{hyperref}

\theoremstyle{plain}

\newtheorem{theorem}{Theorem}[section]
\newtheorem{proposition}[theorem]{Proposition}

\newtheorem{lemma}[theorem]{Lemma}
	 \newtheorem{question}[subsection]{Question}
	  \newtheorem{assumption}[subsection]{Assumption}

\theoremstyle{definition}

\newtheorem{definition}[theorem]{Definition}
\newtheorem{remark}[theorem]{Remark}

\newtheorem{example}[theorem]{Example}

\newcommand\R{\mathbb{R}}
\newcommand\Z{\mathbb{Z}}
\newcommand\N{\mathbb{N}}
\newcommand\C{\mathbb{C}}

\newcommand\eps{\varepsilon}

\dajAUTHORdetails{%
  title = {Sum-avoiding Sets in Groups}, %% please capitalize all significant words
  author = {Terence Tao and Van Vu},
  plaintextauthor = {Terence Tao, Van Vu},
  plaintexttitle = {Sum-avoiding Sets in Groups}, %%  title without math or LaTeX
  keywords = {sum-free sets, sum-avoiding sets, nonstandard analysis},
}   %%% END \dajAUTHORdetails

\dajEDITORdetails{%
   year={2016},
   number={15},
   received={11 March 2016},   % received date: example: 7 January 2017
   published={31 August 2016},  % published date
   doi={10.19086/da.887},       % XXX = number of paper, e.g. da006 for paper#6
}   %%% END \dajEDITORdetails

\begin{document}

\begin{frontmatter}[classification=text]
%% EDITOR: this will force the keywords to appear right after the Abstract.
%%   If the abstract is too long and would force the keywords off the
%%   front page, please comment out % [classification=text] above
%%   This way the keywords will be floated on the bottom of the first page
%%   even though the Abstract spills over to the next page.

%%% AUTHOR: Title goes here.  This line is optional.  You must use it
%%   if title has footnote attached or requires nontrivial typesetting,
%%   e.g., inclusion of linebreaks to force nice layout.
\title{Sum-avoiding sets in  groups} %% please capitalize all significant words

%%% AUTHOR:
%%% List all authors. If you wish, place grant acknowledgements in \thanks.
%%% In brackets include a short tag for each author.
\author[tt]{Terence Tao\thanks{TT is supported by NSF grant DMS-0649473 and by a Simons Investigator Award.}}
\author[vv]{Van Vu\thanks{VV is supported by research grants DMS-0901216 and AFOSAR-FA-9550-09-1-0167.}}

%%% AUTHOR: Abstract goes here
\begin{abstract}
Let $A$ be a finite subset of an arbitrary additive group $G$, and let $\phi(A)$ denote the cardinality of the largest subset $B$ in $A$ that is sum-avoiding in $A$ (that is to say, $b_1+b_2 \not \in A$ for all distinct $b_1,b_2 \in B$).  The question of controlling the size of $A$ in terms of $\phi(A)$ in the case when $G$ was torsion-free was posed by Erd\H{o}s and Moser.  When $G$ has torsion, $A$ can be arbitrarily large for fixed $\phi(A)$ due to the presence of subgroups.  Nevertheless, we provide a qualitative answer to an analogue of the Erd\H{o}s-Moser problem in this setting, by establishing a structure theorem, which roughly speaking asserts that $A$ is either efficiently covered by $\phi(A)$ finite subgroups of $G$, or by fewer than $\phi(A)$ finite subgroups of $G$ together with a residual set of bounded cardinality.  In order to avoid a large number of nested inductive arguments, our proof uses the language of nonstandard analysis.  

We also answer negatively a question of Erd\H{o}s regarding large subsets $A$ of finite additive groups $G$ with $\phi(A)$ bounded, but give a positive result when $|G|$ is not divisible by small primes.
\end{abstract}
\end{frontmatter}

%%% AUTHOR: body of paper starts here
\section{Introduction}

Let $G = (G,+)$ be an additive group (i.e., an abelian group with group operation denoted by $+$).  For $x \in G$, we write $2x := x+x$, $3x := x+x+x$, etc..  If $A,B$ are two subsets of $G$, we say (following \cite{Ruzsa}) that $B$ is \emph{sum-avoiding in $A$} if there do not exist distinct $b_1,b_2 \in B$ such that $b_1+b_2 \in A$.  We then define
\begin{equation}\label{phi-def}
 \phi(A) := \sup \{ |B|: B \subseteq A; B \hbox{ finite and sum-avoiding in } A \}
\end{equation}
where $|B|$ denotes the cardinality of a finite set $B$;
equivalently, $\phi(A)$ is the smallest integer with the property that, given any $\phi(A)+1$ elements of $A$, two of them will also sum to an element of $A$.  In particular $1 \leq \phi(A) \leq |A|$ for non-empty finite $A$.  Informally, sets $A$ with a small value of $\phi(A)$ are ``somewhat sum-closed'' in the sense that sums $a+a'$ of distinct elements of $A$ ``often'' remain in $A$.

To estimate $\phi(A)$ in terms of $|A|$ when $A$ is finite non-empty is an old problem posed by Erd\H{o}s and Moser \cite{erd1965}.  In the case where $G$ is a torsion-free\footnote{A group is \emph{torsion-free} if one has $nx \neq 0$ for any $x \in G \backslash \{0\}$ and any $n=1,2,3,\ldots$.  For instance, a lattice such as $\Z^d$ is torsion-free.} group (without loss of generality one can take $G$ to be the integers $\Z$ in this case), a simple application of Tur\'an's theorem (for graphs with variable degrees) gives 
\begin{equation}\label{phi-log}
\phi(A) \geq \log |A| - O(1)
\end{equation}
for finite non-empty $A$ (see \cite{choi}, as well as an elementary proof in \cite{Ruzsa}).  This bound has been slowly improved in recent years \cite{ssv}, \cite{JD}, \cite{shao}, with the best known bound currently being
$$
\phi(A) \geq \log |A| (\log\log |A|)^{1/2-o(1)}
$$
(see \cite[Corollary 1.3]{shao}).  Equivalently, one has $|A| \leq \exp( \phi(A) / \log^{1/2-o(1)} \phi(A) )$.  In the converse direction, there are examples of arbitrarily large finite sets of integers $A$ for which $\phi(A) = \exp( O( \sqrt{ \log |A| } ) )$; see \cite{Ruzsa}.
It remains a challenging question to reduce the gap between the upper and lower bounds.  See our survey \cite{tv-survey} for more discussion.

In this paper we consider the case when the group $G$ is allowed to have torsion.  Then it is possible for $\phi(A)$ to remain small even when $A$ is arbitrarily large.  For instance, we have the following classification of the additive sets $A$ for which $\phi(A)=1$:

\begin{proposition}[Characterisation of $\phi(A)=1$]\label{easy} Let $A$ be a finite subset of an additive group $G$.  Then $\phi(A)=1$ if and only if one of the following is true:
\begin{itemize}
\item $A=H$, where $H \leq G$ is a subgroup of $G$.
\item $A=H \backslash \{0\}$, where $H \leq G$ is a $2$-torsion subgroup of $G$ (thus $2x=0$ for all $x \in H$).
\item $A = \{b\}$ for some $b \in G$.
\item $A = \{b,0\}$ for some $b \in G$.
\item $A = \{b,0,-b\}$ for some $b \in G$.
\end{itemize}
\end{proposition}

\begin{proof} It is easy to verify that $\phi(A)=1$ in all of the above five cases.  Now suppose that $\phi(A)=1$; then we have
\begin{equation}\label{sum}
b_1 + b_2 \in A \hbox{ whenever } b_1, b_2 \in A \hbox{ and } b_1 \neq b_2
\end{equation}

Suppose first that there exists a non-zero $a \in A$ such that $2a \in A$.  Then from \eqref{sum} we see that the injective map $x \mapsto x+a$ maps $A$ to $A$, and thus must be a bijection on $A$.  This implies that the map $x \mapsto x-a$ is also a bijection on $A$.  Combining these facts with \eqref{sum}, we see that $A$ is closed under addition (since we can shift $b_1, b_2, b_1+b_2$ back and forth by $a$ as necessary to make $b_1, b_2$ distinct).  Since $A$ is finite, every element must have finite order, and then $A$ is then closed under negation, and so $A$ is a subgroup.

It remains to consider the case when 
\begin{equation}\label{2aa}
2a \not \in A \hbox{ for every non-zero } a \in A.
\end{equation}
Suppose now that there exists an element $b \in A$ such that $2b \neq 0$.  We claim that $A$ must then lie in the group generated by $b$.  For if this were not the case, then take an element $a \in A$ which is not generated by $b$, in particular $a \neq 0$.  By iterating \eqref{sum} we see that $a+kb \in A$ for all positive $k$, thus $b$ must have finite order.  In particular, $a+b, a-b \in A$, and by \eqref{sum} again (and the hypothesis $2b \neq 0$) we see that $2a \in A$, contradicting \eqref{2aa}.  Once $A$ lies in the group generated by $b$, it is not hard to see that $A$ must be one of $\{b\}$, $\{b,0\}$, or $\{b,0,-b\}$, simply by using the observation from \eqref{sum} that the map $x \mapsto x+b$ maps $A \backslash \{b\}$ into $A$, together with \eqref{2aa}.

The only remaining case is the $2$-torsion case when $2b=0$ for all $b \in A$.  Then either $A = \{0\}$, or else by \eqref{2aa} $A$ does not contain zero.  In the latter case we observe from \eqref{sum} that $A \cup \{0\}$ is closed under addition and is thus a $2$-torsion group.  The claim follows.
\end{proof}

Now we consider the case when $\phi(A)$ is a little larger than $1$, say $\phi(A) \leq k$ for some moderately sized $k$.
From Proposition \ref{easy} and the easily verified subadditivity property
\begin{equation}\label{subadd}
 \phi(A \cup B) \leq \phi(A) + \phi(B)
\end{equation}
we see that this situation can occur for instance if $A$ is the union of up to $k$ arbitrary finite subgroups, or the union of up to $k-1$ arbitrary subgroups and a singleton.
The main result of this paper is to establish a partial converse to this observation, covering $A$ efficiently by up to $k$ groups, or up to $k-1$ groups and a set of bounded cardinality:

\begin{theorem}[Small $\phi$ implies covering by groups]\label{main} Let $A$ be a finite subset of an additive group $G$ with $\phi(A) \leq k$ for some $k \geq 1$.  Then there exist finite subgroups $H_1,\ldots,H_{m}$ of $G$ with $0 \leq m \leq k$
 such that 
\begin{equation}\label{amm}
|A \backslash (H_1 \cup \ldots \cup H_{m})| \leq C(k)
\end{equation}
and 
\begin{equation}\label{amm-2}
|A \cap H_i| \geq |H_i|/C(k)
\end{equation}
for all $1 \leq i \leq m$.  Here $C(k) > 0$ is a quantity  depending only on $k$ (in particular, it does not depend on $G$ or $|A|$).  If furthermore $m=k$, we may strengthen \eqref{amm} to
$$ A \subseteq H_1 \cup \ldots \cup H_k.$$
\end{theorem}

Note that Proposition \ref{easy} gives the $k=1$ case of this theorem with $C(1)=3$.

Intuitively, the reason for Theorem \ref{main} is as follows.  If $\phi(A)$ is equal to some small natural number $k$ and $A$ is large, then we expect many pairs $a,a'$ in $A$ to sum to another element in $A$.  Standard tools in additive combinatorics, such as the Balog-Szemer\'edi theorem \cite{balog} and Freiman's theorem in an arbitrary abelian group \cite{gr-4}, then should show that $A$ contains a large component that is approximately a \emph{coset progression} $H+P$: the Minkowski sum of a finite group $H$ and a multidimensional arithmetic progression $P$.  Because of bounds such as \eqref{phi-log} that show that $\phi$ becomes large on large torsion-free sets, one expects to be able to eliminate the role of the ``torsion-free'' component $P$ of the coset progression $H+P$, to conclude that $A$ has large intersection with a finite subgroup $H$.  In view of the subadditivity \eqref{subadd}, one heuristically expects $\phi(A)$ to drop from $k$ to $k-1$ after removing $H$ (that is to say, one expects $\phi(A \backslash H) = k-1$), at which point one can conclude by induction on $k$ starting with Proposition \ref{easy} as a base case.  More realistically, one expects to have to replace the conclusion $\phi(A \backslash H) = k-1$ with some more technical conclusion that is not exactly of the same form as the hypothesis $\phi(A)=k$, which makes a direct induction on $k$ difficult; instead, one should expect to have to perform a $k$-fold iteration argument in which one removes up to $k$ subgroups $H_1,\dots,H_m$ from $A$ in turn until one is left with a small residual set $A \backslash (H_1 \cup \dots \cup H_m)$. 

Unfortunately, when the group $G$ contains a lot of torsion, removing a large subgroup $H$ from $A$ can leave one with a residual set with no good additive structure, and in particular with no bounds whatsoever on $\phi(A \backslash H)$.  For instance, suppose that there is an element $x$ of $G \backslash H$ with $2x \in H$, and take $A$ to be the union of $H$ and an \emph{arbitrary} subset of $x+H$.  Then it is easy to see that $\phi(A)$ is at most $2$, but upon removing the large finite group $H$ from $A$ one is left with an arbitrary subset of $x+H$, and in particular $\phi(A \backslash H)$ can be arbitrarily large.

The problem in this example is that the group $H$ is the ``incorrect'' group to try to remove from $A$; one should instead remove the larger group $H' := H + \{0,x\}$, which contains $H$ as an index two subgroup.  The main difficulty in the argument is then to find an algorithm to enlarge an ``incorrect'' group $H$ to a ``correct'' group that absorbs all the relevant ``torsion'' that is present.  This is not too difficult at the start of the iterative argument mentioned above, but becomes remarkably complicated in the middle of the iteration when one has already removed some number of large subgroups $H_1,\dots,H_{m'}$ from the initial set $A$.  A particular technical difficulty comes from the fact that the groups $H_1,\dots,H_{m'}$, as well as the residual set $A \backslash (H_1 \cup \dots \cup H_{m'})$, can have wildly different sizes; in particular, sets which are negligible when compared against one of the $H_i$, could be extremely large when compared against the residual set $A \backslash (H_1 \cup \dots \cup H_{m'})$.  To get around these issues, one needs to ensure some ``transversality'' between these components of $A$, in the sense that the intersection between any of these two sets (or translates thereof) are much smaller than either of the two sets.  This adds an extra layer of complexity to the iterative argument; so much so, in fact, that it becomes very unwieldy to run the argument in a purely finitary fashion.  Instead, we were forced to formulate the argument in the language of nonstandard analysis\footnote{For some prior uses of nonstandard analysis in additive combinatorics, see e.g. \cite{jin} or \cite{gtz}.} in order to avoid a large number of iterative arguments to manage a large hierarchy of parameters (somewhat comparable in complexity to those used to prove the hypergraph regularity lemma, see e.g. \cite{gowers-hypergraph}, \cite{nagle-rodl-schacht}, \cite{nagle-rodl-schacht1}, \cite{tao:hyper}).  One byproduct of this is that our arguments currently provide no bound whatsoever on the quantity $C(k)$ appearing in the above theorem; indeed we expect that if one were to translate the nonstandard analysis arguments back to a standard finitary setting, that the bound obtained on $C(k)$ would be of Ackermann type in $k$ or worse.

In addition to the issue of lack of bounds for $C(k)$, Theorem \ref{main} remains unsatisfactory in another respect, in that it does not fully describe the structure of $A$ inside each of the component groups $H_i$, other than to establish positive density in the sense of \eqref{amm-2}.  On the one hand one does not expect as simple a description of the sets $A \cap H_i$ as in Proposition \ref{easy}, as was already seen in the preceding example when $A$ was the union of a finite group $H$ and an arbitrary subset of a coset $x+H$ with $2x \in H$.  On the other hand, not every dense subset of a finite group has a small value of $\phi$.  For instance, in $\Z/N\Z$, the set $A := \{N/3, \ldots, 2N/3\}$ has $\phi(A) = N/3 + O(1)$, even though $A$ is contained in a finite group of order $O(|A|)$.  See also Proposition \ref{counter}, Proposition \ref{seven}, and Proposition \ref{sidon} below for examples of a non-trivial dense subset of a cyclic group with a small value of $\phi$.    Nevertheless, we are able to sharpen the classification in Theorem \ref{main} when $G$ is a finite group whose order is not divisible by small primes; see Theorem \ref{odd} below.

\subsection{A question of Erd\H{o}s}

In \cite{erd1965}, Erd\H{o}s posed the following question, written here in our notation.

\begin{question}\label{quest}  Let $k$ be a natural number, let $G$ be a finite additive group, and let $A$ be a subset of $G$ with $\phi(A) < k$.  Assume that $|A|$ is sufficiently large depending on $k$.  Does there necessarily exist $a_1,a_2 \in A$ such that $a_1+a_2=0$?
\end{question}

From Proposition \ref{easy} we see that the answer to the question is affirmative when $k=2$.  The case $k=3$ was settled in \cite{ls} (who in fact established the claim for all $|A| \geq 48$).  To our knowledge the cases $k \geq 4$ have remained open in the literature, though in \cite{baltz} a probabilistic construction was given, for any $n$ of subsets $A$ of a finite group of $n$ elements with $\phi(A) = O( \log^2 n )$ such that there were no $a_1,a_2 \in A$ with $a_1+a_2=0$.  

%From Theorem \ref{main} and the easy observation $\phi(A \cap H) \leq \phi(A)$ for any subgroup $H$, we see that to answer Question \ref{quest} it suffices to do so in the regime where $A$ is a dense subset of $G$ (that is to say, one has $|A|/|G| \geq c_k$ for some small $c_k>0$ depending only on $k$).

For $k \geq 5$, we have the following simple negative result:

\begin{proposition}[Counterexample]\label{counter}  Let $n \ge 4$ be a natural number, and set $G$ to be the cyclic group $G := \Z/2^n \Z$.  Let $A \subset G$ be the set
$$ A := \{ (4m+1)2^j \hbox{ mod } 2^n: m \in \Z, 0 \leq j \leq n-2 \}.$$
Thus, for instance, if $n=4$, then $A = \{ 1, 2, 4, 5, 9, 10, 13 \hbox{ mod } 16 \}$. 
Then $\phi(A) = 4$ and $|A| = 2^{n-1}-1$, but there does not exist $a_1,a_2 \in A$ with $a_1+a_2=0$.
\end{proposition}

\begin{proof}  It is easy to see that if $a \in A$, then $-a \not \in A$, and that
$$ |A| = 2^{N-2} + 2^{N-3} + \dots + 1 = 2^{n-1}-1$$
as claimed, and the set $\{ 1, 2, 5, 10 \hbox{ mod } 2^n \}$ is always sum-avoiding in $A$, so $\phi(A) \geq 4$.  The only remaining thing to establish is the upper bound $\phi(A) \leq 4$.  Suppose for contradiction that there existed distinct $a_1,a_2,a_3,a_4,a_5 \in A$ such that the $\binom{5}{2}$ sums $a_i+a_{i'}$ with $1 \leq i < i' \leq 5$ were all outside $A$.  We can write $a_i = (4m_i+1)2^{j_i} \hbox{ mod } 2^n$ with $j_1 \leq \dots \leq j_5$.  If $j_5 > j_1+1$ then $a_5$ is a multiple of $4 \times 2^{j_1}$, and hence $a_1+a_5 = (4m_1+1)2^{j_1}+a_5$ lies in $A$, a contradiction.  Thus $j_1,\dots,j_5$ lie in $\{j_1,j_1+1\}$.  By the pigeonhole principle, we can then find $1 \leq i < i' \leq 5$ such that $j_i = j_{i'}$ and such that $m_i, m_{i'}$ have the same parity.  Note that $j_i=j_{i'}$ cannot equal $n-2$ since $a_i,a_{i'}$ would then both equal $2^{n-2} \hbox{ mod } 2^n$.  But then $a_i+a_{i'}$ is of the form $(4m+1)2^{j_i+1}$, where $m$ is the average of $m_i$ and $m_{i'}$, and so $a_i+a_{i'} \in A$, again a contradiction.
\end{proof}

In particular, the answer to Erd\H{o}s's question is negative for any $k \geq 5$.  For $k=4$, or more generally for any $k \geq 4$ for which $2^{k-1}-1$ is a Mersenne prime, we have an even simpler counterexample:

\begin{proposition}[Mersenne prime counterexample]\label{seven}  Let $k \geq 4$ be such that $p := 2^{k-1}-1$ is prime.  Let $G := (\Z/p\Z) \times H$ for some arbitrary finite group $H$, and let $A := A_0\times H$ where $A_0 := \{ 2^i \hbox{ mod } p: i =0,\dots,k-2\}$.  Then $\phi(A)=k-1$ and $|A| = (k-1)|H|$, but there does not exist $a_1,a_2 \in A$ with $a_1+a_2=0$.
\end{proposition}

\begin{proof}  The only non-trivial claim is that $\phi(A)=k-1$.  From computing sums from the set $A_0 \times \{0\}$ we see that $\phi(A) \geq k-1$.  Suppose for contradiction that there existed distinct $a_1,\dots,a_k \in A$ with all sums $a_i+a_j$ outside $A$.  By the pigeonhole principle we can find $1 \leq i < j \leq k$ such that $a_i, a_j \in \{a\} \times H$ for some $a \in A_0$.  Then $a_i+a_j \in \{2a\} \times H$.  Since $2a$ is also in $A_0$, we obtain a contradiction.
\end{proof}

Thus Erd\H{o}s's question is now resolved for all values of $k$.  But the above examples heavily on the order of the group $G$ being highly divisible by a small prime ($2$ and small Mersenne primes respectively).  In the opposite case, where the order of $G$ is not divisible by any small primes, we have a positive result, which strengthens Theorem \ref{main} by greatly improving the density of $A$ in each subgroup $H_i$:

\begin{theorem}[The case of no small prime divisors]\label{odd}  Let $k$ be a natural number, let $\eps > 0$, and assume that $C_0$ is sufficiently large depending on $k,\eps$.  Let $A$ be a subset of a finite additive group $G$ with $\phi(A) < k$.  Assume that $|G|$ is not divisible by any prime less than $C_0$.  Then there exist subgroups $H_1,\dots,H_m$ of $G$ with $0 \leq m < k$ such that
$$ |A \cap H_i| > (1-\eps) |H_i|$$
for all $i=1,\dots,m$, and
$$ |A \backslash (H_1 \cup \dots \cup H_m)| \leq C_0.$$
Furthermore, if $m=k-1$, then we can take $A \backslash (H_1 \cup \dots \cup H_m)$ to be empty.
\end{theorem}

Applying this theorem with $\eps = 1/2$, we conclude in particular that if $|G|$ is not divisible by any prime less than $C_0$ and $|A| > C_0$, then $A$ must have density greater than $1/2$ in some subgroup $H$ of $G$, and hence there exist $a_1,a_2 \in A \cap H$ with $a_1+a_2 = 0$.  Thus the answer to Erd\H{o}s's question becomes positive if we assume that $|G|$ is not divisible by small primes.

We establish Theorem \ref{odd} in Section \ref{odd-sec}; the argument uses Theorem \ref{main}, as well as a general arithmetic removal lemma of Kr\'al, Serra, and Vena \cite{vena} (which is in turn proven using the hypergraph removal lemma), and also some Fourier analysis.  From Proposition \ref{seven} we see (assuming the existence of infinitely many Mersenne primes!) that the quantity $C_0$ in the above theorem has to at least exponentially growing in $k$.

One could hope to get even more control on the exceptional set $H_i \backslash A$, in the spirit of Proposition \ref{easy}.  But the following example shows that the exceptional set can be somewhat non-trivial in size:

\begin{proposition}\label{sidon}  Let $H$ be a finite group of odd order, and let $B \subset H$ be a Sidon set, that is to say a set $B = \{b_1,\dots,b_m\}$ whose sums $b_i+b_j$, $1 \leq i < j \leq m$, are all distinct.  Then $\phi(H \backslash B) < 4$.
\end{proposition}

It was shown by Erd\H{o}s and Tur\'an \cite{et}, Erd\"os \cite{erdos}, and Chowla \cite{chowla} that there are Sidon subsets of $\{1,\dots,N\}$ of size $(1+o(1)) \sqrt{N}$ for large $N$, which easily implies the existence of Sidon sets in the cyclic group $\Z/N\Z$ of size $\gg \sqrt{N}$.  This shows that the exceptional sets $H_i \backslash A$ in Theorem \ref{odd} can be arbitrarily large.

\begin{proof}  Suppose for contradiction that there exist distinct $a_1,a_2,a_3,a_4 \in A$ whose six sums $a_i+a_j$, $1 \leq i < j \leq 4$ lie outside of $A$, and thus lie in $B$.  Of these six sums, many pairs of sums are forced to be distinct from each other, for instance $a_1+a_2$ is distinct from $a_1+a_3$.  Indeed, the only collisions possible are $a_1+a_2=a_3+a_4$, $a_1+a_3=a_2+a_4$, and $a_2+a_3=a_1+a_4$.  If for instance $a_1+a_2=a_3+a_4$, $a_1+a_3=a_2+a_4$ are both true then $2(a_2-a_3)=0$, which implies $a_2=a_3$ since $|H|$ is odd, a contradiction.  Thus there is at most one collision; by symmetry we may assume then that $a_1+a_2 \neq a_3+a_4$ and $a_1+a_3 \neq a_2+a_4$, which implies that $a_1+a_2,a_3+a_4, a_1+a_3, a_2+a_4$ are all distinct from each other.  But then from the identity
$$ (a_1+a_2) + (a_3+a_4) = (a_1+a_3) + (a_2+a_4)$$
we see that $B$ is not a Sidon set, a contradiction.
\end{proof}

We thank Ben Green for helpful comments and references.
  
\section{Nonstandard analysis formulation}

We quickly review the nonstandard analysis formalism that we need.  We assume the existence of a standard \emph{universe} ${\mathfrak U}$, which is a set that contains all the mathematical objects needed for Theorem \ref{main} such as the group $G$, the set $A$, the natural numbers $\N = \{1,2,\dots\}$, the integers $\Z$ and reals $\R$, etc..  The precise nature of ${\mathfrak U}$ is not important for us, so long as it remains a set. Objects in this universe will be referred to as \emph{standard}, thus for instance elements of $\N$ will be called \emph{standard natural numbers} for emphasis.

We fix a \emph{non-principal ultrafilter} $\alpha \in \beta \N \backslash \N$: a collection of subsets of $\N$ (known as \emph{$\alpha$-large} sets) obeying the following axioms:
\begin{itemize}
\item If $A \subset B \subset \N$ and $A$ is $\alpha$-large, then $B$ is $\alpha$-large.
\item The intersection of any two $\alpha$-large sets is again $\alpha$-large.
\item If $A \subset \N$, then exactly one of $A$ and $\N \backslash A$ is $\alpha$-large.
\item No finite subset of $\N$ is $\alpha$-large.
\end{itemize}
The existence of a non-principal ultrafilter is easily established by Zorn's lemma.

A \emph{nonstandard object} is an equivalence class of a sequence $(x_{\mathfrak n})_{{\mathfrak n} \in \N}$ of standard objects ${\mathfrak n}$, where two sequences $(x_{\mathfrak n})_{{\mathfrak n} \in \N}$ and $(y_{\mathfrak n})_{{\mathfrak n} \in \N}$ are equivalent if one has $x_{\mathfrak n} = y_{\mathfrak n}$ for an $\alpha$-large set of ${\mathfrak n}$.  This equivalence class will be denoted $\lim_{{\mathfrak n} \to \alpha} x_{\mathfrak n}$, and called the \emph{ultralimit} of the sequence $x_{\mathfrak n}$.  We define a \emph{nonstandard natural number} to be the ultralimit of a sequence of standard natural numbers; similarly define the concept of a nonstandard integer, nonstandard real, etc..  The set of nonstandard natural numbers is denoted ${}^* \N$; similarly define ${}^* \Z$, ${}^* \R$, etc..  We can embed $\N$ in ${}^* \N$ by identifying any standard natural number $n$ with its nonstandard counterpart $\lim_{{\mathfrak n} \to \alpha} n$; similarly for the integers and reals.

All the usual arithmetic operations and relations on the standard natural numbers $\N$ extend to nonstandard natural numbers.  For instance the sum of two nonstandard numbers is given by the formula
$$ \lim_{{\mathfrak n} \to \alpha} x_{\mathfrak n} + \lim_{{\mathfrak n} \to \alpha} y_{\mathfrak n} := \lim_{{\mathfrak n} \to \alpha} (x_{\mathfrak n} + y_{\mathfrak n}),$$
and the ordering on nonstandard natural numbers is given by defining
$$ \lim_{{\mathfrak n} \to \alpha} x_{\mathfrak n} < \lim_{{\mathfrak n} \to \alpha} y_{\mathfrak n}$$
whenever $x_{\mathfrak n} < y_{\mathfrak n}$ for an $\alpha$-large set of ${\mathfrak n}$.  One can verify that all the usual (first-order) laws of arithmetic continue to hold for the nonstandard natural numbers; this is a special case of {\L}os's theorem, which we will not state here.  Similarly for the nonstandard integers and reals.

Given a sequence $X_1,X_2,\dots$ of standard sets $X_{\mathfrak n}$, we define the \emph{ultraproduct} $\prod_{{\mathfrak n} \to \alpha} X_{\mathfrak n}$ to be the collection of ultralimits $\lim_{{\mathfrak n} \to \alpha} x_{\mathfrak n}$ with $x_{\mathfrak n} \in X_{\mathfrak n}$ for an $\alpha$-large set of ${\mathfrak n}$.  An ultraproduct of standard sets will be called a \emph{nonstandard set}\footnote{Nonstandard sets are also known as \emph{internal sets} in the nonstandard analysis literature.}; an ultraproduct of standard finite sets will be called a \emph{nonstandard finite set}; and an ultraproduct of standard additive groups will be called a \emph{nonstandard additive group}.  Note that nonstandard additive groups are also additive groups in the usual (or ``external'') .  Also, any Cartesian product $A \times B$ of two nonstandard finite sets $A, B$ is again (up to some canonical isomorphism which we ignore) a nonstandard finite set.  Similarly for Boolean operations such as $A \cup B$, $A \cap B$, or $A \backslash B$.

Given a nonstandard finite set $A = \prod_{{\mathfrak n} \to \alpha} A_{\mathfrak n}$, we define its \emph{nonstandard cardinality} $|A|$ to be the nonstandard natural number
$$ |A| := \lim_{{\mathfrak n} \to \alpha} |A_{\mathfrak n}|.$$
Thus for instance, if $N$ is a nonstandard natural number, then the nonstandard finite set $\{ n \in {}^* \N: n \leq N\}$ has nonstandard cardinality $|A|=N$.  This notion of cardinality obeys the usual rules such as $|A \times B| = |A| |B|$ and $|A \cup B| \leq |A| + |B|$; we will use such rules without further comment below.  The notion of nonstandard cardinality of a nonstandard finite  set $A$ agrees with the usual (or ``external'') notion of cardinality when $|A|$ is a standard natural number, but if this is not the case then $A$ will be infinite in the external sense.

Now we recall the nonstandard formulation of asymptotic notation.  A nonstandard number $x$ (in ${}^* \N$, ${}^* \Z$, or ${}^* \R$) is said to be \emph{bounded} if one has $|x| \leq C$ for some standard real $C$, and \emph{unbounded} otherwise.  Thus for instance a nonstandard natural number is bounded if and only if it standard.  Since standard numbers cannot be unbounded, we abbreviate ``unbounded nonstandard number'' as ``unbounded number''.  If $|x| \leq \eps$ for every standard real $\eps>0$, we say that $x$ is \emph{infinitesimal}.  If $X, Y$ are nonstandard numbers with $Y > 0$, we write $X=O(Y)$, $X \ll Y$, or $Y \gg X$ if $X/Y$ is bounded, and $X=o(Y)$ if $X/Y$ is infinitesimal.  We caution that if one has sequences $X_n, Y_n$ of nonstandard numbers indexed by a standard natural number $n$, and $X_n \ll Y_n$, then no uniformity in $n$ is assumed\footnote{If $X_n$ could be extended in an ``internal'' fashion to nonstandard $n$, and $X_n \ll Y_n$ for all such $n$, then one could make the constants $C_n$ uniform in $n$ by the saturation properties of nonstandard analysis, but we will not need to exploit saturation in this paper.}; one can only say that $|X_n| \leq C_n Y_n$ for some standard $C_n$ that can depend on $n$.  On the other hand, if we know that $X_n = o(Y_n)$, then we have the uniform bound $|X_n| \leq \eps Y_n$ for \emph{all} standard $n$ and $\eps>0$.

A major advantage of the nonstandard formulation for us is the ability to easily define various properties of (nonstandard) finite sets, whose counterparts in the finitary world would either have to be somewhat vague or informal, or else involve a number of auxiliary quantitative parameters that one would then need to carefully keep track of.  We record the main properties we will need here:

\begin{definition}[Nonstandard concepts]\label{nonst}  Let $G$ be a nonstandard additive group, and let $A, B$ be nonstandard finite subsets of $G$.
\begin{itemize}
\item[(i)] We say that $A$ is \emph{small} if its cardinality $|A|$ is a standard natural number, and \emph{large} otherwise.  (Note that $A$ is small if and only if it is (externally) finite.)  We say that $A,B$ have \emph{comparable size} if $|A| \ll |B| \ll |A|$.
\item[(ii)]  We say that $A$ is \emph{somewhat sum-closed} if $\phi(A)$ is (standardly) finite. 
\item[(iii)] We write
$$ A+B := \{ a+b: a \in A, b \in B \}$$ 
for the Minkowski sum of $A$ and $B$, and say that $A$ has \emph{bounded doubling} if $|A+A| \ll |A|$.  For instance, any coset $x+H$ of a nonstandard finite group $H$ is of bounded doubling.
\item[(iv)] For any standard natural number $l$, we write $lA = A + \dots + A$ for the $l$-fold iterated sumset of $A$.  (
Although it can be done, we will not need to define $lA$ for unbounded $l$.)
\item[(v)] We say that $A$ \emph{avoids} $B$ if $|A \cap B| = o(|B|)$, and \emph{occupies} $B$ if it does not avoid $B$ (or equivalently, if $|A \cap B| \gg |B|$).  For instance, if $A$ and $B$ are groups, $A$ avoids $B$ when $A \cap B$ has (externally) infinite index in $B$, and $A$ occupies $B$ when $A \cap B$ has (externally) finite index in $B$.  See also Figure \ref{fig:com}.
\item[(vi)]  We say that $A$ and $B$ are \emph{commensurate} if they occupy each other, and \emph{transverse} if they avoid each other.  
\item[(vii)] A sequence $A_n$ of sets indexed by standard natural numbers $n$ is said to be an \emph{wandering sequence} if for every standard natural numbers $n,m$, we have $(A_n+A_m) \cap (A_n+A_{m'}) = \emptyset$ for all but finitely many $m'$.  For instance, if $x_n+H$ is a sequence of disjoint cosets of a single nonstandard finite group $H$, then the $x_n+H$ form an wandering sequence.
\item[(viii)] We say that $A$ has the \emph{wandering property} if the sequence $nA$ is an wandering sequence, or equivalently that $nA \cap n'A = \emptyset$ whenever $n'$ is sufficiently large depending on $n$.  For example, a coset $x+H$ of a nonstandard finite group $H$ has the wandering property if and only if $x$ has infinite order in $G/H$, or equivalently if $nx \not \in H$ for any standard $n$.  As another example, if $G = {}^* \Z$ are the nonstandard integers and $N$ is an unbounded natural number, then the nonstandard interval $\{ n \in {}^* \Z: N \leq n < 2N \}$ has the wandering property, but the nonstandard interval $\{ n \in {}^* \Z: 1 \leq n \leq N\}$ does not.  
\end{itemize}
All of the above notions are understood to be restricted to nonstandard finite sets; for instance, if we say that $A$ is large, it is understood that $A$ is also nonstandard finite.
\end{definition}

\begin{figure} [t]
\centering
\includegraphics[width=10cm]{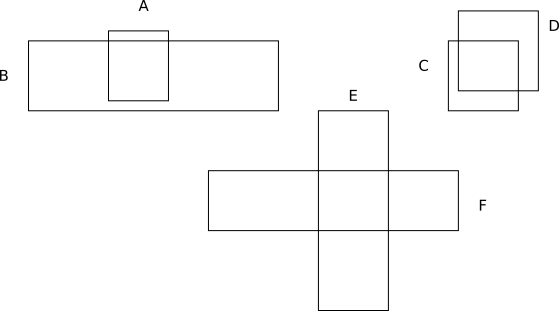}
\caption{In this schematic (and ``not-to-scale'') diagram, $B$ occupies $A$, but $A$ avoids $B$; $C$ and $D$ are commensurate; and $E$ and $F$ are transverse.}
\label{fig:com}
\end{figure}

\begin{remark}
We caution that the assertion ``$A$ is a large subset of $B$'' does \emph{not} mean that $A$ has positive density in $B$ in the sense that $|A| \gg |B|$; rather, it means instead that $|A|$ is unbounded.  (The notion of positive density is captured instead by the assertion ``$A$ is a \emph{commensurate} subset of $B$''.)  For instance, if $N$ is an unbounded natural number, then $\{ n \in {}^* \N: n \leq N \}$ is a large subset of $\{ n \in {}^* \N: n \leq N^2\}$, despite having infinitesimal density in the latter set.  As noted in the introduction, one key difficulty with the inductive argument is that one has to simultaneously deal with multiple sets that are all large, but can be of wildly differing sizes, making it tricky to transfer density information from one set to another.  For instance, if $B$ avoids $A$, and $A'$ is a large subset of $A$, we cannot automatically conclude that $B$ avoids $A'$ (for instance, consider the counterexample $B=A'=\{ n \in {}^* \N: n \leq N \}$ and $A=\{ n \in {}^* \N: n \leq N^2\}$).
\end{remark}

To illustrate some of the above concepts, we give an easy application of the pigeonhole principle that will be used repeatedly in the sequel:

\begin{lemma}[Pigeonhole principle]\label{pigeon}  Let $A$ be a somewhat sum-closed subset of a nonstandard additive group $G$, and write $k := \phi(A)$.  Let $A_1,\dots,A_{k+1}$ be subsets of $A$, with $A_1,\dots,A_k$ large, and $A_{k+1}$ nonstandard finite (we permit $A_{k+1}$ to be small).  Then there exist $1 \leq i < j \leq k+1$ such that
$$ | \{ (a_i,a_j) \in A_i \times A_j: a_i+a_j \in A \}| \geq \left(\frac{2}{k(k+1)}-o(1)\right) |A_i| |A_j|.$$
In particular,
$$ | \{ (a_i,a_j) \in A_i \times A_j: a_i+a_j \in A \}| \gg |A_i| |A_j|.$$
\end{lemma}

\begin{proof}  We can assume that $A_{k+1}$ is non-empty, as the claim is trivial otherwise.  Consider the set $A_1 \times \dots \times A_{k+1}$ of tuples $(a_1,\dots,a_{k+1})$ with $a_i \in A_i$ for all $i=1,\dots,k+1$.  This is a nonstandard finite set of cardinality $|A_1| \dots |A_{k+1}|$.  For any $1 \leq i < j \leq k+1$, the subset of tuples with $a_i=a_j$ has cardinality at most $|A_1| \dots |A_{k+1}|/|A_i|$, which is $o( |A_1| \dots |A_{k+1}| )$ since $A_i$ is large.  Thus for a proportion $1-o(1)$ of the tuples $(a_1,\dots,a_{k+1})$ in $A_1 \times \dots \times A_{k+1}$, the $a_1,\dots,a_{k+1}$ are distinct elements of $A$.  Since $k = \phi(A)$, we conclude that for each such tuple, one of the sums $a_i+a_j$ lies in $A$.  By the pigeonhole principle, we can thus find $1 \leq i < j \leq k+1$ such that for a proportion $\frac{1}{\binom{k+1}{2}} - o(1)$ of the tuples, $a_i+a_j \in A$, and the claim follows.
\end{proof}

Theorem \ref{main} is a consequence of (and is in fact equivalent to) the following nonstandard formulation.

\begin{theorem}\label{main-nonst} Let $A$ be a somewhat sum-closed subset of a nonstandard additive group $G$, and set $k := \phi(A)$.
Then there exist nonstandard finite subgroups $H_1,\ldots,H_{m}$ of $G$ with $0 \leq m \leq k$
 such that $A \backslash (H_1 \cup \ldots \cup H_{m})$ is small and $A$ occupies each of the $H_1,\dots,H_m$.  Furthermore, if $m=k$, then
$A \subset H_1 \cup \dots \cup H_k$.
\end{theorem}

Let us now see how Theorem \ref{main-nonst} implies Theorem \ref{main}; the converse implication is not needed here and is left to the interested reader.  This is a routine application of the transfer principle in nonstandard analysis, but for the convenience of the reader we provide a self-contained argument.  Suppose for contradiction that Theorem \ref{main} failed.  Carefully negating the quantifiers, we conclude that there is a standard $k \geq 1$, a sequence $G_{\mathfrak n}$ of standard additive groups and finite subsets $A_{\mathfrak n}$ of $G_{\mathfrak n}$ with the following property for any ${\mathfrak n}$: $\phi(A_{\mathfrak n}) \leq k$, and there does \emph{not} exist finite subgroups $H_{1,{\mathfrak n}},\ldots,H_{m,{\mathfrak n}}$ of $G_{\mathfrak n}$ with $0 \leq m \leq k$ 
 such that 
$$
|A_{\mathfrak n} \backslash (H_{1,{\mathfrak n}} \cup \ldots \cup H_{m,{\mathfrak n}})| \leq {\mathfrak n}
$$
and  
$$
|A_{{\mathfrak n}} \cap H_{i,{\mathfrak n}}| \geq |H_{i,{\mathfrak n}}|/{\mathfrak n}
$$
for all $1 \leq i \leq m$, and such that
$$ A_{{\mathfrak n}} \subseteq H_{1,{\mathfrak n}} \cup \ldots \cup H_{k,{\mathfrak n}}$$
in the case $m=k$.

Let $G := \prod_{{\mathfrak n} \to \alpha} G_{{\mathfrak n}}$ denote the ultraproduct of the $G_{\mathfrak n}$, and similarly define
$A := \prod_{{\mathfrak n} \to \alpha} A_{{\mathfrak n}}$.  Then $G$ is a nonstandard additive group, and $A$ is a nonstandard finite subset of $G$.  If $A$ contained a subset $\{x_1,\dots,x_{k+1}\}$ of cardinality exactly $k+1$ that was sum-avoiding in $A$, then by writing each $x_i$ as an ultralimit $x_i = \lim_{{\mathfrak n} \to \alpha} x_{i,{\mathfrak n}}$ it is easy to see that for an $\alpha$-large set of ${\mathfrak n}$, one has $\{x_{1,{\mathfrak n}},\dots,x_{k+1,{\mathfrak n}}\}$ a sum-avoiding subset of $A_{\mathfrak n}$ of cardinality exactly $k+1$, contradicting the hypothesis $\phi(A_{\mathfrak n}) \leq k$.  Thus, such sets $\{x_1,\dots,x_{k+1}\}$  do not exist, and so $\phi(A) \leq k$.  Applying Theorem \ref{main-nonst}, we can now find nonstandard finite groups $H_1,\dots,H_m$ with $0 \leq m \leq k$ such that
$$ |A \backslash (H_1 \cup \ldots \cup H_{m})| \leq C$$
and
$$ |A \cap H_i| \geq |H_i| / C$$
for some standard $C$ and all $1 \leq i \leq m$, and such that $A \subset H_1 \cup \dots \cup H_k$ if $m=k$.  Writing each $H_i = \prod_{{\mathfrak n} \to \alpha} H_{i,{\mathfrak n}}$ as the ultraproduct of finite subgroups of $G_{\mathfrak n}$, we conclude for an $\alpha$-large set of ${\mathfrak n}$ that
$$ |A_{\mathfrak n} \backslash (H_{1,\mathfrak n} \cup \ldots \cup H_{m,{\mathfrak n}})| \leq C$$
and
$$ |A_{\mathfrak n} \cap H_{i,{\mathfrak n}}| \geq |H_{i,{\mathfrak n}}| / C$$
for all $1 \leq i \leq m$, with $A_{\mathfrak n} \subset H_{1,\mathfrak n} \cup \ldots \cup H_{k,{\mathfrak n}}$ when $m=k$.  But this contradicts the construction of the $A_{\mathfrak n}$ for ${\mathfrak n} > C$, and the claim follows.

It remains to establish Theorem \ref{main-nonst}.  This theorem will be derived from the following claim.

\begin{theorem}\label{main-2}  Let $A$ be a somewhat sum-closed subset of a nonstandard finite group $G$, and let $m$ be a standard non-negative integer.  Then there exists $0 \leq m' \leq m$ and large pairwise transverse subgroups $H_1,\dots,H_{m'}$ of $G$ with the following properties, with $A'$ denoting the residual set $A' := A \backslash \bigcup_{i=1}^{m'} H_i$:
\begin{itemize}
\item[(i)] For each $i=1,\dots,m'$, $A$ occupies $H_i$, but $A'$ avoids all cosets $x+H_i$ of $H_i$.
\item[(ii)] If the set $A'$ is large and $1 \leq i \leq m'$, then all cosets $x+H_i$ of $H_i$ avoid $A'$.  In other words, by (i), $A'$ is transverse to all cosets of $H_1,\dots,H_{m'}$.
\item[(iii)]  If $m' < m$, then the set $A'$ is small.
\end{itemize}
\end{theorem}

\begin{figure} [t]
\centering
\includegraphics[width=5cm]{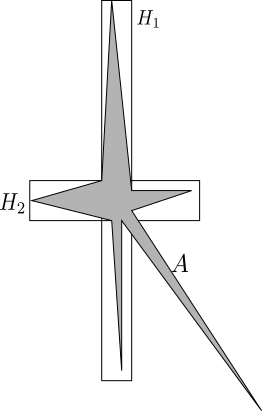}
\caption{A schematic depiction of a situation described by Theorem \ref{main-2} in which $2 = m' < m$, and $A'$ (the portion of $A$ outside of $H_1 \cup H_2$) is somewhat large, but transverse to all cosets of $H_1$ or $H_2$.  Note that $A'$ is permitted to be significantly smaller than $H_1$, $H_2$, or $A$.}
\label{fig:mainpic}
\end{figure}

See Figure \ref{fig:mainpic}.

Let us now see how Theorem \ref{main-2} implies Theorem \ref{main-nonst}.  Let $A$ be a somewhat sum-closed subset of a nonstandard finite group $G$, and set $k := \phi(A)$.  We apply Theorem \ref{main-2} with $m=k$ to obtain $0 \leq m' \leq k$ and large pairwise transverse $H_1,\dots,H_{m'}$ of $G$ with the properties (i)-(iii).  If $m' < k$, then we are done from property (iii), so we may assume $m=k$.  If $A'$ is empty then we are again done, so suppose for contradiction that $A'$ contains an element $a_{k+1}$.

Note that for $i=1,\dots,k$, $A \cap H_i$ is commensurate with the large set $A_i$ and is thus itself large.  Applying Lemma \ref{pigeon} with $A_i := A \cap H_i$ for $i=1,\dots,k$ and $A_{k+1} := \{a_{k+1}\}$, we can find $1 \leq i < j \leq k+1$ such that
$$ | \{ (a_i,a_j) \in A_i \times A_j: a_i+a_j \in A \}| \gg |A_i| |A_j| \gg |H_i| |H_j|.$$
There are two cases, depending on whether $j \leq k$ or $j = k+1$.  First suppose that $j \leq k$.  For any $1 \leq l \leq k$ with $l \neq i$, $H_l$ is transverse to $H_i$, and so
$$ | \{ (a_i,a_j) \in A_i \times A_j: a_i+a_j \in H_l \}| \leq |A_j| |H_i \cap H_l| = o( |H_i| |H_j| ).$$
Similarly if $l \neq j$ instead of $l \neq i$.  By the triangle inequality, we conclude that
$$ | \{ (a_i,a_j) \in A_i \times A_j: a_i+a_j \in A' \}| \gg |H_i| |H_j|.$$
However, since $A'$ avoids all cosets $a_i+H_j$ of $H_j$, we have
$$ | \{ (a_i,a_j) \in A_i \times A_j: a_i+a_j \in A' \}| = o( |A_i| |H_j| ) = o( |H_i| |H_j|),$$
giving the required contradiction.

Now suppose that $j=k+1$, thus
$$ | \{ a_i \in A_i: a_i+a_{k+1} \in A \}| \gg |H_i|$$
and hence
$$ |A \cap (H_i + a_{k+1})| \gg |H_i|.$$
For any $1 \leq l \leq k$ distinct from $i$, the group $H_l$ is transverse to the group $H_i$ and hence to the coset $H_i+a_{k+1}$; also $H_i$ is disjoint from $H_i+a_{k+1}$ by construction.  Finally, $A'$ does not occupy $H_i+a_{k+1}$ by (i).  Combining these facts using the triangle inequality, we conclude that
$$ |A \cap (H_i + a_{k+1})| = o(|H_i|),$$
again giving the required contradiction.  This concludes the derivation of Theorem \ref{main-nonst} (and hence Theorem \ref{main}) from Theorem \ref{main-2}.

It remains to establish Theorem \ref{main-2}.   Before we begin the main argument, we establish some preliminary facts about somewhat sum-closed sets that we will use several times in the sequel.  We begin with a somewhat technical lemma that asserts (roughly speaking) that somewhat sum-closed sets $A$ are avoided by at least one member of an wandering sequence $B_n$, provided that some auxiliary hypotheses are satisfied.

\begin{lemma}[Avoiding an wandering sequence]\label{tc}  Let $A$ be a somewhat sum-closed subset of a nonstandard additive group $G$.  Let $B_n$ be an wandering sequence of subsets of $G$, indexed by standard natural numbers $n$.  For each $n$, suppose we have a large subset $A'_n$ of $A$, such that for all but finitely many standard natural numbers $m$, we have
\begin{equation}\label{trick}
 |(z + A'_n) \cap (A \backslash A'_n) \cap (B_n+B_m)| = o(|A'_n|)
\end{equation}
for all $z \in G$ (or equivalently, $(A \backslash A'_n) \cap (B_n+B_m)$ avoids all translates of $A'_n$).  Then there must exist $n$ such that $A'_n$ is avoided by $B_n$.
\end{lemma}

\begin{proof}  Set $k := \phi(A)$.  Suppose for contradiction that $A'_n$ is occupied by $B_n$ for each $n$.  Then $A'_n$ has comparable size to $A'_n \cap B_n$ for each $n$ (though the comparability need not be uniform in $n$); in particular, $A'_n \cap B_n$ is large.  For fixed $n$, we claim that
$$ |A'_n \cap (B_n+B_m)| \leq \frac{1}{k(k+1)} |A'_n \cap B_n|$$
(say) for all but finitely many standard $m$.  For, if this inequality failed for infinitely many $m$, then by using the wandering nature of the $B_n$ one could find an infinite sequence $m_1,m_2,\dots$ with
$$ |A'_n \cap (B_n+B_{m_i})| \geq \frac{1}{k(k+1)} |A'_n \cap B_n|$$
for all $i$, with the $B_n + B_{m_i}$ pairwise disjoint, but this contradicts the fact that $A'_n$ and $A'_n \cap B_n$ have comparable size.

From the above claim and \eqref{trick} (and the comparable size of $A'_n$ and $A'_n \cap B_n$) we have
$$ |(z+A'_n) \cap A \cap (B_n+B_m)| \leq \left(\frac{1}{k(k+1)}+o(1)\right) |A'_n \cap B_n|$$
for all $z \in G$, assuming $m$ is sufficiently large depending on $n$.  Iterating, we can find an increasing sequence $n_1,\dots,n_{k+1}$ such that
\begin{equation}\label{bah}
|(z+A'_{n_i}) \cap A \cap (B_{n_i} + B_{n_j})| \leq \left(\frac{1}{k(k+1)}+o(1)\right) |A'_{n_i} \cap B_{n_i}|
\end{equation}
for all $1 \leq i < j \leq k+1$ and $z \in G$.  On the other hand, by applying Lemma \ref{pigeon} to the sets $A_i := A'_{n_i} \cap B_{n_i}$, we can find $1 \leq i < j \leq k+1$ such that
\begin{equation}\label{so}
\begin{split}
& | \{ (a_i,a_j) \in (A'_{n_i} \cap B_{n_i}) \times (A'_{n_j} \cap B_{n_j}): a_i+a_j \in A \}| \\
&\quad \geq (\frac{2}{k(k+1)}-o(1)) |A'_{n_j} \cap B_{n_i}| |A'_{n_j} \cap B_{n_j}|.
\end{split}
\end{equation}
Note that for $(a_i,a_j)$ contributing to the left-hand side of \eqref{so}, we have
$$ a_i + a_j \in (a_j + A'_{n_i}) \cap A \cap (B_{n_i} + B_{n_j}).$$
Summing in $a_j$ using \eqref{bah} and then in $a_i$, we can thus upper bound the left-hand side of \eqref{so} by
$$ \left(\frac{1}{k(k+1)}+o(1)\right) |A'_{n_i} \cap B_{n_i}| |A'_{n_j} \cap B_{n_j}|,$$
giving the required contradiction.
\end{proof}

Lemma \ref{tc} leads to the following useful consequence.

\begin{lemma}[Avoiding wandering sets]\label{tf}  Let $G$ be a nonstandard additive group.  Let $B$ be a subset of $G$ with the wandering property. Let $A$ be a somewhat sum-closed subset of $G$, and let $A'$ be a large subset of $A$, such that 
\begin{equation}\label{smal}
|(A \backslash A') \cap lB| = o(|A'|)
\end{equation}
for all but finitely many natural numbers $l$.  Then $B$ avoids $A'$.
\end{lemma}

This lemma is already interesting in the case $A=A'$, in which case it asserts that a set $B$ with the wandering property must necessarily avoid any large somewhat sum-closed set $A$.  Another illustrative case is when $A'$ and $B$ are contained in some subgroup $H$ of $G$, and $A \backslash A'$ is contained in some subgroup $K$ with $H \cap K = \emptyset$, in which case \eqref{smal} is automatically satisfied.  A key point in this lemma is that $A'$ is permitted to be significantly smaller than $A$, so long as the technical condition \eqref{smal} can be verified.

\begin{proof}  
Suppose for contradiction that $B$ occupied $A'$, so $A' \cap B$ is commensurate to $A'$ and is in particular large.  Applying Lemma \ref{pigeon} with $A_i := A' \cap B$ for $i=1,\dots,\phi(A)+1$, we have
$$ |\{ (a_1,a_2) \in (A' \cap B) \times (A' \cap B): a_1+a_2 \in A \}| \gg |A' \cap B|^2 \gg |A'|^2$$
and thus
$$ |\{ (a_1,a_2) \in A': a_1+a_2 \in A' \cap 2B \}| \gg |A'|^2.$$
Each element of $A' \cap 2B$ contributes at most $|A'|$ to the left-hand side, and hence
$$ |A' \cap 2B| \gg |A'|$$
and hence $2B$ occupies $A'$.  Iterating this argument, we see that $2^n B$ occupies $A'$ for every standard natural $n$.  

By hypothesis, the sequence $2^n B$ is n wandering sequence.  From \eqref{smal} we have for all but finitely many $m$ that
$$
 |(z + A') \cap (A \backslash A') \cap (2^n B + 2^m B)| = o(|A'|)
$$
for all $z \in G$.  Applying Lemma \ref{tc}, we conclude that there is a standard natural $n$ such that $A'$ is avoided by $2^n B$, giving the required contradiction.
\end{proof} 

%
%To end this section, let us consider a special case of Lemma \ref{tf}. Assume that $A'$ and $B$ are subsets of a subgroup $H$, and $A\backslash A'$ is a subset of a subgroup $K$ where $H \cap K = \{ 0 \}$. Then \eqref{smal} holds trivially and thus we conclude that 
%$B$ avoid  $A'$. 
%
%
%\begin{corollary}\label{tf1}  Let $H$ be a nonstandard additive group.  Let $B$ be a subset of $H$ with the wandering property and  $A'$ be a somewhat sum-closed subset of $H$. Then $B$ avoids $A'$. \end{corollary}
%
%\begin{proof} 
%Let $K$ be a copy of $H$ and $A^{''} $ be a copy of $A'$ in $K$. Apply the above argument to the product $G:=K \times  H$ and the set $A:= A' \cup A^{''}$. 
%\end{proof} 
%
%
%{\bf To Terry. It looks like Corollary \ref{tf1}  implies Lemma \ref{tf} without \eqref{smal}, as if $B$ avoids the whole set then it avoids any big subset. It them may simplify the argument to remove $P$ somewhat ás one does  not need to check \eqref{smal}.  } 

\section{Main argument}

We now prove Theorem \ref{main-2}.  We achieve this by induction on $m$.  When $m=0$ the claim is trivially true, so we now assume inductively that Theorem \ref{main} has been proven for some value of $m \geq 0$, and seek to establish the theorem for $m+1$.
The reader may initially wish to follow the argument below in the case $m=0$, which is significantly simpler than the general case.

Let $A$ be somewhat sum-closed subset of a nonstandard finite group $G$.  By the induction hypothesis, we can find $0 \leq m' \leq m$ and pairwise transverse large subgroups $H_1,\dots,H_{m'}$ of $G$ obeying the properties (i)-(iii) of Theorem \ref{main-2}.

If $m' < m$, then from property (iii) we obtain Theorem \ref{main-2} for $A$ and $m+1$, so we may assume that $m'=m$.  As before we set
$$ A' := A \backslash \bigcup_{i=1}^m H_i.$$
If $A'$ is small then we are done, so let us assume that $A'$ is large.  In particular, by (i) and (ii), all of the cosets of $H_1,\dots,H_m$ are transverse to $A'$.

Applying Lemma \ref{pigeon} with all $A_i$ equal to $A'$, we conclude that
$$ |\{ (a_1,a_2) \in A' \times A': a_1+a_2 \in A \}| \gg |A'|^2.$$
For any $1 \leq l \leq m$, all cosets $-a_2+H_l$ avoid $A'$, and hence
$$ |\{ (a_1,a_2) \in A' \times A': a_1+a_2 \in H_l \}| = o(|A'|^2).$$
By the triangle inequality, we conclude that
$$ |\{ (a_1,a_2) \in A' \times A': a_1+a_2 \in A' \}| \gg |A'|^2.$$
By the Cauchy-Schwarz inequality, this implies that
$$ |\{ (a_1,a_2,a_3,a_4) \in A' \times A' \times A' \times A': a_1+a_2 = a_3 + a_4 \}| \gg |A'|^3.$$
Applying the Balog-Szemer\'edi theorem (see e.g. \cite[Section 2.5]{tao-vu}, and transferring from standard analysis to nonstandard analysis in the usual fashion), we conclude that $A'$ has a commensurate subset $A''$ of bounded doubling.  

The next step is to invoke Freiman's theorem.  Define a \emph{nonstandard coset progression} to be a nonstandard finite subset of $G$ of the form
$$ H + P $$
where $H$ is a nonstandard finite subgroup, and $P$ is a nonstandard finite set of the form
$$ P = \{ a + n_1 v_1 + \dots + n_r v_r: 1 \leq n_i \leq N_i \hbox{ for all } i=1,\dots,r\}$$
for some standard non-negative integer $r$ (called the \emph{rank} of the progression), some base point $a \in G$ and generators $v_1,\dots,v_r$ of $G$, and some nonstandard natural numbers $N_1,\dots,N_r$.  
Applying Freiman's theorem in an arbitrary abelian group (see \cite{gr-4}) to the commensurate subset $A''$ of $A'$ of bounded doubling (again transferring from standard analysis to the nonstandard analysis), we can find a nonstandard coset progression $H+P$ that is commensurate with $A'$.  Among all such progressions, we select a nonstandard coset progression $H+P$ commensurate with $A'$ of minimal rank $r$; the existence of such an $H+P$ comes from the (external) principle of infinite descent.
If the coset progression is \emph{improper} in the sense that there is a relation of the form
\begin{equation}\label{never}
 n_1 v_1 + \dots + n_r v_r \in H 
\end{equation}
for some nonstandard $n_i = O(N_i)$ for $i=1,\dots,r$, not all zero, then we can contain the nonstandard coset progression $H+P$ in a commensurate nonstandard coset progression of lower rank; see e.g. \cite[Corollary 1.19]{tv-john}.  As this contradicts minimality, we conclude that $H+P$ is proper, in the sense that there are no non-trivial relations of the form \eqref{never}.  In a similar spirit, none of the dimensions $N_1,\dots,N_r$ are bounded, since otherwise we may remove the generators $v_1,\dots,v_r$ corresponding to the bounded dimensions and use the pigeonhole principle to locate a lower rank nonstandard coset progression that is still commensurate with $A'$, contradicting minimality.

Next, suppose there is a relation of the form
\begin{equation}\label{nanv}
 na + n_1 v_1 + \dots + n_r v_r \in H
\end{equation}
for some standard integer $n$ and some nonstandard $n_i = O(N_i)$ for $i=1,\dots,r$, with $n,n_1,\dots,n_r$ not all zero.  As $H+P$ is proper, $n$ is non-zero. Then by \cite[Corollary 1.19]{tv-john} again (placing $H+P$ inside a symmetric improper coset progression with generators $a,v_1,\dots,v_r$), we may contain $H+P$ in a commensurate coset progression of the same rank as $H+P$, with the additional property that the base point $a$ is of the form 
\begin{equation}\label{nvr}
a = n_1 v_1 + \dots + n_r v_r 
\end{equation}
for some nonstandard integers $n_i = O(N_i)$.  Thus we may assume either that there is no relation of the form \eqref{nanv}, or else that the base point $a$ has the form \eqref{nvr}.

Suppose first that the former case occurs, that is to say there is no relation of the form \eqref{nanv}.  In particular, $H+P$ has the wandering property.  If any of the $H_i$ occupy $l(H+P)$, then by covering $l(H+P)$ by a bounded number of translates of $H+P$ we see that some coset of $H_i$ occupies $H+P$, and a simple averaging argument using the commensurability of $H+P$ and $A'$ then shows that some coset of $H_i$ occupies $A'$, contradicting Theorem \ref{main-2}(ii).  Thus we see that $A \backslash A'$ does not occupy any of the $l(H+P)$; in particular, since $l(H+P)$ has comparable size to $H+P$ and hence to $A'$, we have
$$ |(A \backslash A') \cap l(H+P)| = o(|A'|).$$
Applying Lemma \ref{tf}, we conclude that $A'$ is avoided by $H+P$, a contradiction. Thus we may assume that there is a relation of the form \eqref{nanv}, and hence by the preceding discussion we may assume that the base point $a$ is of the form \eqref{nvr}.

For any standard real numbers $0 < \alpha_i \leq \beta_i$ and signs $\epsilon_i \in \{-1,1\}$, consider the set
$$ P' =  \{ n_1 v_1 + \dots + n_r v_r:  \alpha_i N_i \leq \epsilon_i n_i \leq \beta_i N_i\ \forall 1 \leq i \leq r \}.$$
One observes (from the properness of $H+P$) that $H+P'$ has the wandering property.  By the preceding arguments, we know that 
$$ |(A \backslash A') \cap l(H+P')| = o(|A'|).$$
Applying Lemma \ref{tf} again, we conclude that $A$ is avoided by all of the $H+P'$.  If the rank $r$ is positive, we can use the unboundedness of the dimensions $N_1,\dots,N_r$ as well as the relation \eqref{nvr} to cover $H+P$ by a bounded number of sets of the above form $H+P'$, plus an error of cardinality at most $\eps |H+P|$ for any given standard $\eps>0$.  Sending $\eps$ to zero (and recalling that $H+P$ and $A'$ are of comparable size), we would conclude that $A$ is avoided by $H+P$, a contradiction.  We conclude that the rank $r$ must vanish, which by \eqref{nanv} means that $a=0$, and so $H+P$ is simply $H$.

To summarise the progress so far, we have located a nonstandard finite subgroup $H$ of $G$ that is commensurate with $A'$, so in particular $A$ occupies $H$.  Since $A'$ is large, $H$ is also.  Since no coset of the $H_1,\dots,H_m$ occupies $A'$, none of the $H_i$ occupy $H$; conversely, since $A'$ avoids all the cosets of $H_i$, $H$ also avoids $H_i$.  Thus $H$ is transverse to each of the $H_1,\dots,H_m$.

As discussed in the introduction, the group $H$ may be the ``incorrect'' group to use to close the induction, and one may have to replace $H$ with a larger commensurate group.  As a starting point, we have

\begin{proposition}\label{fc}  The group $H$ is a subgroup of a commensurate group $H_{m+1}$, with the property that $A' \backslash H_{m+1}$ avoids all the cosets of $H_{m+1}$.
\end{proposition}

\begin{proof} By the preceding discussion, $H$ is large and commensurate to $A'$, and $A \backslash A'$ avoids all the cosets of $H$; in particular
$$ |(z+A') \cap (A \backslash A') \cap (x+y+H)| = o(|A'|)$$
for any $x,y,z \in G$.  By Lemma \ref{tc} (with $A'_n=A'$ and arguing by contradiction with $B_n$ set equal to various cosets of $H$), $A'$ is occupied by only a standardly finite number of cosets of $H$; equivalently, since $A'$ is of comparable size to $H$, it occupies at most finitely many cosets of $H$.  Let $x+H$ be one of these cosets.  If all the dilates $nx+H$ for $n$ a natural number are disjoint, then $x+H$ has the wandering property , and hence by Lemma \ref{tf}, $A'$ avoids $x+H$, a contradiction.  Thus all of the (standardly finitely mahy) cosets $x+H$ that $A'$ occupies has bounded torsion in $G/H$, in the sense that $nx+H = H$ for some standard $n$.  If we let $H_{m+1}$ be the group generated by these cosets, then $H_{m+1}$ thus contains $H$ as a commensurate subgroup (since $H_{m+1}/H$ is a finitely generated abelian group with all generators of finite order, and so $H_{m+1}/H$ is small).  Since $A' \backslash H_{m+1}$ avoids all the cosets of $H$, it also avoids all the cosets of the commensurate group $H_{m+1}$.  The claim follows.
\end{proof}

Note that as $H$ is large and transverse to $H_1,\dots,H_m$, the commensurate group $H_{m+1}$ is also large and transverse to $H_1,\dots,H_m$.  In particular, the $H_1,\dots,H_{m+1}$ are all large and pairwise transverse.  By the above proposition and inductive claim of Theorem \ref{main-2}(i), the set $A' \backslash H_{m+1} = A \backslash (H_1 \cup \dots \cup H_{m+1})$ avoids all the cosets of the $H_1,\dots,H_{m+1}$, but $A$ occupies each of the $H_1,\dots,H_{m+1}$.

We are almost done, except that we do not have the required claim of Theorem \ref{main-2}(ii) for $m+1$; that is to say, if $A' \backslash H_{m+1}$ is large, we have not established that $A' \backslash H_{m+1}$ avoids all the cosets of the $H_1,\dots,H_{m+1}$.  In the $m=0$ case, $A=A'$ is commensurate with $H_1$, and this claim follows from the already established property that $A' \backslash H_1$ avoids all the cosets of $H_1$. However, for general $m$, some additional argument (in the spirit of the proof of Theorem \ref{fc}) is needed, in which each of the $H_1,\dots,H_{m+1}$ is replaced with larger commensurate groups.  To run this argument, it is convenient to introduce some additional notation relating to this family $H_1,\dots,H_{m+1}$ of large, pairwise transverse groups.

\begin{definition}  For any $I \subset \{1,\dots,m+1\}$, write $H_I$ for the subgroup $H_I := \bigcap_{i \in I} H_i$ (this subgroup may be trivial), with the convention $H_\emptyset = G$.  If $d$ is a natural number, a \emph{subspace}\footnote{Since it is possible for $H_I$ to equal $H_J$ for distinct $I,J$, we should strictly speaking define a subspace to be a \emph{pair} $(x+H_I, I)$ rather than just $x+H_I$, otherwise concepts such as the codimension of the subspace would be undefined.  However, we shall abuse notation and use $x+H_I$ rather than $(x+H_I,I)$ to denote a subspace.} of $G$ of \emph{codimension} $d$ to be a coset $x+H_I$ of one of the $H_I$ with $|I|=d$ (note that no such subspaces exist if $d>m+1$).  If the codimension is one, we call the subspace a \emph{hyperplane}, that is to say a hyperplane is a coset $x+H_i$ of one of the $H_1,\dots,H_{m+1}$.   A subspace $x+H_I$ is said to be \emph{aperiodic} if  $nx+H_I$ is occupied by one of the $H_1,\dots,H_{m+1}$ for some natural number $n$, and \emph{periodic} otherwise.  Thus for instance $G = H_\emptyset$ (the unique subspace of codimension $0$) is aperiodic for vacuous reasons, and any aperiodic subspace of codimension at least one has the wandering property.  We say that a subspace $x+H_I$ is \emph{captured} by another subspace $y+H_J$ if $x+H_I \subset y+H_J$ and $I \supset J$.

A \emph{core} is a set of the form $C = H'_1 \cup \dots \cup H'_{m+1}$, where each $H'_i$ is a group containing $H_i$ as a commensurate subgroup.  If $C$ is a core, a \emph{$C$-residual set} is a set of the form $A \backslash C'$ where $C'$ is a core containing $C$.  A subspace $x+H_I$ is said to be \emph{$C$-involved} it has codimension at least one, and there exists a large $C$-residual set $A'$ that is occupied by $x+H_I$.  
\end{definition}

\begin{example}  Let $H_1$ be a nonstandard finite subgroup of $G$, and let $x$ be an element of $G \backslash H_1$ such that $2x \in H_1$.  Let $A$ be the union of $H_1$ and a large subset $A'$ of $x+H_1$.  Then $A'$ is $H_1$-residual and $x+H_1$ is an $H_1$-involved hyperplane.  However, if one replaces the core $H_1$ with the larger core $H_1+\{0,x\}$, then there are no $H_1+\{0,x\}$-involved subspaces.  More generally, the basic strategy of the arguments below is to keep collecting involved subspaces into larger and larger cores until we find a core so large that no further involved subspaces appear.
\end{example}
\begin{flushleft}

\end{flushleft}
We then have the following more complicated variant of Proposition \ref{fc}.

\begin{proposition}\label{fc-complex} Let $1 \leq d \leq m+2$.
\begin{itemize}
\item[(i)]  Let $0 \leq d' \leq d$, and let $y + H_J$ be an aperiodic subspace of codimension $d'$.  Then there exists a core $C$ and an (externally) finite family ${\mathcal F}_{d,d',y+H_J}$ of subspaces of codimension $d'+1$, such that every $C$-involved subspace $x+H_I$ of codimension $d$ that is captured by $y+H_J$, is also captured by one of the subspaces in ${\mathcal F}_{d,d',y+H_J}$.  In particular, if $d'=d$, then ${\mathcal F}_{d,d',y+H_J}$ is empty as $x+H_I$ cannot be captured by a codimension $d+1$ subspace.
\item[(ii)]  There exists a core $C = H'_1 \cup \dots \cup H'_{k+1}$, with the property that there are no $C$-involved subspaces of codimension $d$.
\end{itemize}
\end{proposition}

\begin{proof}  We establish (i) and (ii) together by downwards induction on $d$.  The claims (i) and (ii) are vacuously true for $d=m+2$ as there are no subspaces of that codimension, so suppose inductively that $1 \leq d \leq m+1$, and that the claims (i), (ii) are already established for $d+1$.  Thus, by (ii), we have a core $C = H'_1 \cup \dots \cup H'_{m+1}$ such that there are no $C$-involved subspaces of codimension $d+1$.

Next, we prove (i).  Let $I \subset \{1,\dots,m+1\}$ have cardinality $d$.  By taking unions in $I$, it will suffice to locate a finite family ${\mathcal F}_{d,d',y+H_J,I}$ of subspaces of codimension $d'+1$, such that every $C$-involved subspace $x+H_I$ captured by $y+H_J$ is also captured by one of the subspaces in ${\mathcal F}_{d,d',y+H_J,I}$.

Let $x+H_I$ be a $C$-involved subspace that is captured by $y+H_J$, thus there is a large $C$-residual set $\tilde A := A \backslash C''$ that is occupied by $x+H_I$, where $C'' := H''_1 \cup \dots \cup H''_m$ is a core containing $C$.  We first claim that $x+H_I$ is periodic.  To see this, suppose for contradiction that $x+H_I$ is aperiodic, then it has the wandering property.  Furthermore, for any natural number $n$, $nx+H_I$ is avoided by $H_i$ and hence by $H''_i$ for each $1 \leq i \leq k+1$.   In particular
$$ |(A \backslash \tilde A) \cap n(x+H_I)| = 0.$$
Applying Lemma \ref{tf}, we conclude that $x+H_I$ avoids $\tilde A$, a contradiction.  Thus $x+H_I$ is periodic.

Suppose for contradiction that the required claim fails; that is to say, we suppose

\begin{assumption}\label{as} Given any (externally) finite family ${\mathcal F}$ of subspaces of codimension $d'+1$, there exists an involved subspace $x+H_I$ captured by $y+H_J$ that is not captured by any subspace in ${\mathcal F}$.  
\end{assumption}

Let $x_1+H_I,\dots,x_n+H_I$ be some finite sequence of subspaces captured by $y+H_J$.  Thus to each $x_{n'}+H_I$, $1 \leq n' \leq n$, there is a large residual set $A'_i := A \backslash C''_{n'}$ which is occupied by $x_{n'}+H_I$, where $C''_{n'} := H''_{1,n'} \cup \dots \cup H''_{m+1,n'}$ is a core containing $C$.  We claim that we can find an $C$-involved subspace $x_{n+1}+H_I$ captured by $y+H_J$ and distinct from each of the $x_{n'}+H_I$ for $1 \leq n' \leq n$, and such that
$$ |(z + A'_{n'}) \cap (A \backslash A'_{n'}) \cap (x_{n'}+x_{n+1}+H_I)| = o( |A'_{n'}| )$$
for each $1 \leq n' \leq n$ and $z \in G$.  To see this, observe first that for any $j \in J$, the set $x_n'+x_{n+1}+H_I = x_{n'}+H_J + x_{n+1}+H_J$ lies in $y+H_J+y+H_J = 2y + H_J$, which avoids $H''_{j,n'}$ since $y+H_J$ is aperiodic.  Thus
$$ |(z + A'_{n'}) \cap H''_j \cap (x_{n'}+x_{n+1}+H_I)| = 0$$
for such $j$.  Next, suppose that $1 \leq j \leq m+1$ with $j \not \in I$.  Then by the hypothesis (ii), none of the cosets of $H_j \cap H_I = H_{I \cup \{j\}}$ occupy the $C$-residual set $A'_{n'}$, and so by the triangle inequality none of the cosets of $H''_j \cap H_I$ do either.  By translation, none of the cosets of $H''_j \cap H_I$ occupy $z+A'_{n'}$, and in particular
$$ |(z + A'_{n'}) \cap H''_{n'} \cap (x_{n'}+x_{n+1}+H_I)| = o(|A'_{n'}|).$$
Next, suppose that $j \in I \backslash J$, so that $H_I$ is a subgroup of $H''_j \cap H_J$.  Then as long as $x_{n+1}+H_I$ does not lie in the set $(-x_{n'} + H''_j) \cap (y + H_J)$, which is the union of finitely many subspaces of codimension $d'+1$, we have
$$ |(z \cap A'_{n'}) \cap H''_j \cap (x_{n'}+x_{n+1}+H_I)| = 0$$
for this value of $j$.  Finally, as $x_{n'}+H_I$ is periodic, it must lie in some periodic hyperplane $w_{n'} + H_{i_{n'}}$; as $y+H_J$ is aperiodic, $i_{n'}$ must lie outside of $J$.  As long as $x_{n+1}+H_I$ does not lie in $(w_{n'} + H_{i_{n'}}) \cap (y+H_J)$, which is a subspace of codimension $d'+1$, it must be distinct from $x_{n'}+H_I$.
Putting all this together, we have identified a finite collection of subspaces of codimension $d'+1$, such that for any $C$-involved $x_{n+1}+H_I$ outside of these subspaces, $x_{n+1}+H_I$ is distinct from the $x_{n'}+H_I$ for $n'=1,\dots,n$, and that
$$ |(z + A'_{n'}) \cap (A \backslash A'_{n'}) \cap (x_{n'}+x_{n+1}+H_I)| = o( |A'_{n'}| )$$
for each $1 \leq n' \leq n$ and $z \in G$.  The existence of such a $x_{n+1}+H_I$ follows from Assumption \eqref{as}, giving the required claim.

Iterating the claim, we can find an infinite sequence $x_n+H_I$ of disjoint $C$-involved subspaces captured by $y+H_J$, such that the associated $C$-residual sets $A'_n = A \backslash C_n$ occupied by $x_n+H_i$ obey the bound
$$ |(z + A'_{n'}) \cap (A \backslash A'_{n'}) \cap (x_{n'}+x_n+H_I)| = o( |A'_{n'}| )$$
for all $1 \leq n' < n$ and $z \in G$.  Applying Lemma \ref{tf} to the wandering sequence $B_n := x_n+H_I$, there must exist some $n$ such that $A'_n$ is avoided by $x_n+H_i$, a contradiction.  This establishes (i).

Since the sum $H'_i + H''_i$ of two groups $H'_i, H''_i$ that both contain $H_i$ as a commensurate subgroup is again a group containing $H_i$ as a commensurate subgroup, we see that for any two cores $C', C''$, there exists another core $C'''$ that contains both (in particular, any $C'''$-involved subspace is also $C'$-involved and $C''$-involved).  If we iterate (i) starting from $d'=0$ and $y+H_J = G$, using (i) to split up any aperiodic subspace encountered into a finite number of higher codimension pieces, and combining all cores encountered together by the preceding observation, we conclude that there exists a core $C$ and a finite family ${\mathcal F}_d$ of \emph{periodic} subspaces of codimensions ranging between $1$ and $d$, such that every $C$-involved subspace $x+H_I$ of codimension $d$ is captured by one of the subspaces in ${\mathcal F}_d$.

Note that every periodic subspace is contained in a periodic hyperplane $x + H_i$, which is in turn contained in a core (by setting $H'_i$ to be the group generated by $H_i$ and $x$, and setting all other $H'_j$ equal to $H_j$).  Combining all the cores together as before, we may thus find a core $C'$ that contains $C$ as well as every subspace in ${\mathcal F}_d$.  Then there cannot be any $C'$-involved subspace $x+H_I$ of codimension $d$ (as such subspaces are both $C$-involved and lie outside of every subspace in ${\mathcal F}_d$), giving the claim (ii) as required.
\end{proof}

Applying the $d=1$ case of Proposition \ref{fc-complex}(ii), we can locate a core $C = H'_1 \cup \dots \cup H'_{m+1}$ with the property that there are no $C$-involved hyperplanes.  Thus, if we set $A'' := A \backslash (H'_1 \cup \dots \cup H'_{m+1})$ and assume $A''$ is large, then for any $i=1,\dots,m$, all cosets of $H_i$ avoid $A''$, and hence all cosets of the commensurable $H'_i$ avoid $A''$.  Also, since $A'$ already avoided all cosets of $H_i$, the smaller set $A''$ also avoids all cosets of the commensurable group $H'_i$.  Finally, since $H_1,\dots,H_{m+1}$ were large and pairwise transverse, the groups $H'_1,\dots,H'_{m+1}$ are also.  This (finally) gives the $m+1$ case of Theorem \ref{main-2}, closing the induction and then giving Theorem \ref{main-nonst} and Theorem \ref{main} as corollaries.

\begin{remark}  A significant portion of the above arguments seem to be extendible\footnote{We thank Ben Green for this observation} to the non-abelian setting, for instance by using the results in \cite{bgt} as a substitute for those in \cite{gr-4}.  In particular, it is reasonable to conjecture that if $A$ is a finite subset of a (not necessarily abelian) group $G = (G,\cdot)$, with the property that for any distinct $a_1,\dots,a_{k+1} \in A$, there exist distinct $i,j \in \{1,\dots,k+1\}$ such that $a_i a_j \in A$, then $A$ should be commensurate with a finite subgroup $H$ of $G$ in the sense that $|A \cap H| \geq c(k) |A|, c(k) |H|$ for some $c(k)>0$ depending only on $k$.  However there appear to be some technical difficulties in transferring some portions of the argument; for instance, the assertion that any translate of a (symmetric) coset progression can be contained in a slightly larger (symmetric) coset progression does not hold for the nonabelian analogue of a coset progression, namely a coset nilprogression, and this leads to some complications which we have not been able to resolve.  We will not pursue these issues further here.
\end{remark}

\section{Groups of order not divisible by small primes}\label{odd-sec}

We now prove Theorem \ref{odd}.  Fix $k$ and $\eps$; we may assume that $k \geq 3$ as the claim follows from Proposition \ref{easy} otherwise.  Let $\eps_1>0$ be a sufficiently small quantity depending on $k,\eps$ to be chosen later.  Let $C_0$ be a quantity to be chosen later that is sufficiently large depending on $\eps,\eps',k$.  Here we will revert back to standard analysis, avoiding all nonstandard notation.  In particular, we now use the terminology $X \ll_{k,\eps} Y$, $Y \gg_{k,\eps} X$, or $X = O_{k,\eps}(Y)$ to denote a bound of the form $|X| \leq C(k,\eps) Y$, where $C(k,\eps)$ is a quantity that depends on $k,\eps$ but not on $C_0,\eps'$.  Similarly for $X \ll_k Y$, $X \gg_k Y$, etc..

Let $A$ be a subset of a finite group $G$ with $\phi(A) < k$, and with $|G|$ not divisible by any prime less than $C_0$.  By applying Theorem \ref{main}, and discarding any group $H_i$ of order less than $C_0^{1/2}$, we can find subgroups $H_1,\dots,H_m$ of $G$ with $0 \leq m < k$ such that
$$
|A \cap H_i| \gg_{k} |H_i|$$
and
$$ |H_i| \geq C_0^{1/2}$$
for all $i=1,\dots,m$, and such that
$$ 
|A \backslash (H_1 \cup \dots \cup H_m)| \ll_{k} C_0^{1/2}.$$
Furthermore if $m=k-1$ we can take $A \backslash (H_1 \cup \dots \cup H_m)$ to be empty.
If we have $|A \cap H_i| > (1-\eps) |H_i|$ for all $i$ then we are done, so suppose there is an $i=1,\dots,m$ such that $|A \cap H_i| \leq (1-\eps) |H_i|$.  We abbreviate $H_i$ as $H$, thus
\begin{equation}\label{ah}
|H \backslash A|, |A \cap H| \gg_{k,\eps} |H| \geq C_0^{1/2}.
\end{equation}
As the order of $H$ divides the order of $G$, we see that $|H|$ is not divisible by any prime less than $C_0$.
Meanwhile, the property $\phi(A) < k$ implies that $\phi(A \cap  H) < k$.  Thus, the only variables $x_1,\dots,x_k \in A \cap H$ and $x_{ij} \in H \backslash A$ for $1 \leq i < j \leq k$ that obey the system of $\binom{k}{2}$ linear equations
$$ x_i + x_j - x_{ij} = 0$$
for all $1 \leq i < j \leq k$, are those for which two of the $x_i,x_j$ are equal; thus the number of solutions to this system with the indicated constraints is at most $\binom{k}{2} |A \cap H|^{k-1}$, which is smaller than $\delta |H|^k$ for any $\delta>0$ if $C_0$ is sufficiently large depending on $k,\delta$.

One can write the linear system more compactly as
$$ M x = 0$$
where $M$ is a certain $\binom{k}{2} \times (k+\binom{k}{2})$ matrix with integer entries, and $x$ is a $k+\binom{k}{2}$-dimensional vector whose entries are $x_1,\dots,x_k$ and $x_{ij}$ for $1 \leq i < j \leq k$.  For instance, for $k=3$ we would have the system
\begin{equation}\label{match}
\begin{pmatrix}
1 & 1 & 0 & -1 & 0 & 0 \\
1 & 0 & 1 & 0 & -1 & 0 \\
0 & 1 & 1 & 0 & 0 & -1
\end{pmatrix}
\begin{pmatrix} x_1 \\ x_2 \\ x_3 \\ x_{12} \\ x_{13} \\ x_{23} \end{pmatrix} = \begin{pmatrix} 0 \\ 0 \\ 0 \\ 0 \\ 0 \\ 0 \end{pmatrix}.
\end{equation}
We apply the arithmetic removal lemma from\footnote{See also \cite{csv} for another proof of this theorem.} \cite[Theorem 1]{vena} (assuming $C_0$ sufficiently large depending on $k,\eps_1'$) to conclude that we can remove at most $\eps_1 |H|$ elements from $A \cap H$ to form a new set $A'$, and at most $\eps_1 |H|$ elements from $H \backslash A$ to form a new set $B'$, with the property that the system $Mx=0$ has \emph{no} solutions with $x_1,\dots,x_k \in A'$ and $x_{ij} \in B'$ for $1 \leq i < j \leq k$.  In particular, setting $x_1=\dots=x_k=a$ and $x_{ij} = 2a$, we conclude that there are no $a \in A'$ for which $2a \in B'$.  As $|H|$ is odd, the map $a \mapsto 2a$ is a bijection on $H$, and we conclude that for all but at most $2\eps_1 |H|$ elements $a$ of $A \cap H$, we have $2a \in A \cap H$, and similarly with $H \backslash A$.  Thus
\begin{equation}\label{ahah2}
\frac{1}{|H|} \sum_{x \in H} |1_{A \cap H}(x) - 1_{A \cap H}(2x)| \leq 4 \eps_1,
\end{equation}
where $1_{A \cap H} \colon H \to \{0,1\}$ is the indicator function of $A \cap H$.

We now introduce the Fourier transform $\hat 1_{A \cap H} \colon \hat H \to \C$ of $1_{A \cap H}$, where the Pontryagin dual group $\hat H$ is defined as the group of all homomorphisms $\xi \colon x \mapsto \xi \cdot x$ from $H$ to $\R/\Z$, and the Fourier transform $\hat f \colon H \to \C$ of any function $f \colon H \to \C$ is given by the formula
$$ \hat f(\xi) := \frac{1}{|H|} \sum_{x \in H} f(x) e^{-2\pi i \xi \cdot x}.
$$
From \eqref{ahah2} and the triangle inequality we see that
$$ |\hat 1_{A \cap H}(2\xi) - \hat 1_{A \cap H}(\xi)| \leq 4 \eps_1$$
for any $\xi \in \hat H$, and hence on iterating
\begin{equation}\label{soj}
 |\hat 1_{A \cap H}(2^j\xi) - \hat 1_{A \cap H}(\xi)| \leq 4 j\eps_1
\end{equation}
for any natural number $j$.

If $\xi$ is a non-zero element of $\hat H$ such that $|\hat 1_{A \cap H}(\xi)| \geq \eps_1^{1/4}$, then from \eqref{soj} we see that
$$ |\hat 1_{A \cap H}(2^j \xi)| \geq \frac{1}{2} \eps_1^{1/4} $$
for any natural number $j \leq \eps_1^{-3/4}/8$.  But since $\hat H$ is isomorphic to $H$ (see e.g. \cite[Chapter 4]{tao-vu}), the order of $\xi$ is not divisible by any prime less than $C_0$.  Thus, if $C_0$ is sufficiently large depending on $\eps_1$, then all the $2^j \xi$ with $j \leq \eps_1^{-1/4}/8$ are distinct.  In particular from the Plancherel identity we have
$$ \sum_{j \leq \eps_1^{-3/4}/8} |\hat 1_{A \cap H}(2^j \xi)|^2 \leq 1.$$
This contradicts the previous bound if $\eps_1$ is small enough.  We conclude that $A \cap H$ is Fourier-uniform in the sense that
\begin{equation}\label{ah-1}
 \sup_{\xi \in \hat H \backslash \{0\}} |\hat 1_{A \cap H}(\xi)| < \eps_1^{1/4}.
\end{equation}
Since the Fourier coefficients of $1_{H \backslash A}$ at non-zero frequencies $\xi$ are the negative of those of $1_{A \cap H}$, we also have
\begin{equation}\label{ah-2}
 \sup_{\xi \in \hat H \backslash \{0\}} |\hat 1_{H \backslash A}(\xi)| < \eps_1^{1/4}.
\end{equation}

Next, we observe that the linear system $Mx=0$ is of ``complexity one'' in the sense of \cite{gt-primes}, which roughly speaking means that the solution count to this system is controlled by the size of Fourier coefficients; this observation was already implicit in the work of Balog \cite{balog}.  More precisely, we have

\begin{proposition}[Complexity one]\label{one}  For functions $f_1,\dots,f_k \colon H \to \R$ and $f_{ij} \colon H \to \R$ for $1 \leq i < j \leq k$, define the $k + \binom{k}{2}$-linear form
$$ \Lambda(f_1,\dots,f_k, (f_{ij})_{1 \leq i < j \leq k}) := \frac{1}{|H|^k} \sum_{x_1,\dots,x_k \in H} \left(\prod_{i=1}^k f_i(x_i)\right) \left(\prod_{1 \leq i < j \leq k} f_{ij}(x_i+x_j)\right).$$
Suppose that the $f_1,\dots,f_k$ and $f_{ij}$ all take values in $[-1,1]$.  Then for any $1 \leq i \leq k$ we have
\begin{equation}\label{fo}
 |\Lambda(f_1,\dots,f_k, (f_{ij})_{1 \leq i < j \leq k})| \leq \sup_{\xi \in \hat H} |\hat f_i(\xi)|
\end{equation}
and similarly for any $1 \leq i < j \leq k$ we have
\begin{equation}\label{fo-2}
 |\Lambda(f_1,\dots,f_k, (f_{ij})_{1 \leq i < j \leq k})| \leq \sup_{\xi \in \hat H} |\hat f_{ij}(\xi)|
\end{equation}
\end{proposition}

\begin{proof}  We begin with \eqref{fo-2}.  We shall just prove this claim for $ij=12$, as the general case follows from appropriate permutation of indices.  By the triangle inequality, it suffices to show that
$$
\frac{1}{|H|^2} \left|\sum_{x_1,x_2 \in H} \left(\prod_{i=1}^k f_i(x_i)\right) \left(\prod_{1 \leq i < j \leq k} f_{ij}(x_i+x_j)\right)\right| \leq \sup_{\xi \in \hat H} |\hat f_{12}(\xi)|
$$
for any choice of $x_3,\dots,x_k \in H$.  But once one fixes such choices, one can rewrite the left-hand side as
$$
\frac{1}{|H|^2} \left|\sum_{x_1,x_2 \in H} f(x_1) g(x_2) f_{12}(x_1+x_2)\right|$$
for some functions $f,g \colon H \to [-1,1]$ whose exact form is not important to us.  By the Fourier inversion formula, this can be written as
$$ |\sum_{\xi \in H} \hat f(\xi) \hat g(\xi) \hat f_{12}(-\xi)|.$$
But from the Plancherel identity we have
$$ \sum_{\xi \in H} |\hat f(\xi)|^2, \sum_{\xi \in H} |\hat g(\xi)|^2 \leq 1,$$
and the claim now follows from the Cauchy-Schwarz or H\"older inequalities.

Now we show \eqref{fo}.  By symmetry we may take $i=1$.  We make the change of variables $x_i = x_2 + h_i$ for all $i=3,\dots,k$.  Then by the triangle inequality, it suffices to show that
$$
\frac{1}{|H|^2} \left|\sum_{x_1,x_2 \in H} \left(\prod_{i=1}^k f_i(x_i)\right) \left(\prod_{1 \leq i < j \leq k} f_{ij}(x_i+x_j)\right)\right| \leq \sup_{\xi \in \hat H} |\hat f_{1}(\xi)|
$$
for any choice of $h_3,\dots,h_k \in H$, with the understanding that $x_i = x_2+h_i$ for $i=3,\dots,k$.  But once one fixes the $h_1,\dots,h_k$, the left-hand side can be written as
$$
\frac{1}{|H|^2} \left|\sum_{x_1,x_2 \in H} f_1(x_1) g(x_2) h(x_1+x_2)\right|$$
for some functions $g,h \colon H \to [-1,1]$ whose exact form is not important to us.  Using the Fourier inversion formula as before, we obtain the claim.
\end{proof}

Let $\sigma := |A \cap H|/|H|$ denote the density of $A$ in $H$, thus from \eqref{ah} we have
\begin{equation}\label{sis}
 \sigma, 1-\sigma \gg_{k,\eps} 1.
\end{equation}
We split $1_{A \cap H} = \sigma + (1_{A \cap H}-\sigma)$ and $1_{H \backslash A} = 1-\sigma + (1_{H \backslash A}-(1-\sigma))$, and observe from \eqref{fo}, \eqref{fo-2} that the Fourier coefficients of $1_{A \cap H}-\sigma$ and $1_{H \backslash A}-(1-\sigma)$ are bounded in magnitude by $\eps_1^{1/4}$.  Applying Proposition \ref{one} many times, we conclude that
\begin{equation}\label{sos}
  \Lambda(1_{A \cap H},\dots,1_{A \cap H}, (1_{H \backslash A})_{1 \leq i < j \leq k}) = \sigma^k (1-\sigma)^{\binom{k}{2}} + O_k(\eps_1^{1/4}).
	\end{equation}
On the other hand, since $\phi(A \cap H) < k$, the only contribution to the left-hand side comes from when two of the $x_i,x_j$ are equal, which gives the upper bound
\begin{equation}\label{sus}
  \Lambda(1_{A \cap H},\dots,1_{A \cap H}, (1_{H \backslash A})_{1 \leq i < j \leq k}) \leq \binom{k}{2} |H|^{-1} \ll_k C_0^{-1/2}.
	\end{equation}
The estimates \eqref{sis}, \eqref{sos}, \eqref{sus} lead to the desired contradiction by choosing $\eps_1$ small and $C_0$ large. This proves Theorem \ref{odd}.

%%% AUTHOR:
%%% Bibliography goes here. Note that the arXiv cannot process bibtex
%%% or biber bibliographies.  Example of acceptable bibliograpy format:

%% AUTHOR: You can generate such a bibliography from a .bib file by 
%% running pdflatex/bibtex/pdflatex/pdflatex and then pasting the .bbl file
%% between \begin{thebibliography} and \end{bibliography}

%%% AUTHOR: Include a short description of each author following the
%%% structure below. Use the same short tags used previously.  
%%% Use \imageat{} and \imagedot{} instead of "@" and "." in
%%% email addresses-this replaces the symbols with graphics to avoid 
%%% e-mail address harvesting from the .pdf file
\begin{dajauthors}
\begin{authorinfo}[tt]
  Terence Tao\\
  Department of Mathematics, UCLA\\
  405 Hilgard Ave\\
	Los Angeles, CA 90095, USA\\
  tao\imageat{}math.ucla.edu\\
	\url{https://www.math.ucla.edu/~tao}
\end{authorinfo}

\begin{authorinfo}[vv]
  Van Vu\\
	Department of Mathematics, Yale University\\
	New Haven, CT 06520, USA \\
	van.vu\imageat{}yale.edu\\
	\url{http://campuspress.yale.edu/vanvu/cv/}
\end{authorinfo}

\end{dajauthors}

\end{document}